\documentclass[reqno]{amsproc}

\usepackage{graphicx}

\usepackage[dvipdfmx]{hyperref}

\hypersetup{
colorlinks=true,
linkcolor=blue,
anchorcolor=blue,
citecolor=blue
}

\usepackage{amssymb}
\usepackage{amsfonts}
\usepackage{amsmath}
    
\newtheorem{theorem}{Theorem}[section]
\newtheorem{lemma}[theorem]{Lemma}

\theoremstyle{definition}
\newtheorem{definition}[theorem]{Definition}

\newtheorem{remark}[theorem]{Remark}

\numberwithin{equation}{section}

\begin{document}

\title{Geometric analysis on manifolds with ends}
\author{Alexander Grigor'yan}
\address{Department of Mathematics University of Bielefeld 33501,  Bielefeld, Germany}
\thanks{Partially supported by SFB 1283 of the German Research Council} 
\email{grigor@math.uni-bielefeld.de}

\author{Satoshi Ishiwata}
\address{Department of Mathematical Sciences Yamagata University, Yamagata 990-8560, Japan}
\thanks{Partially supported by JSPS, KAKENHI 17K05215} 
\email{ishiwata@sci.kj.yamagata-u.ac.jp}

\author{Laurent Saloff-Coste}
\address{Department of Mathematics Cornell University, Ithaca, NY, 14853-4201, USA}
\thanks{Partially supported by NSF grant DMS--1707589}
\email{lsc@math.cornell.edu}


\maketitle

\section*{Contents}
1. Introduction

2. The state of the art

3. Manifold with ends with oscillating volume functions

References

\section{Introduction}
\label{introduction}
In this survey article, we discuss some recent progress on geometric analysis on manifold with ends.
In the final section, we construct manifolds with ends with oscillating volume functions which may turn out to 
have a different heat kernel estimates from those provided by known results.

Throughout the history of geometric analysis, manifolds with ends  have appeared in several contexts. 
For example, 
Cai \cite{Cai}, Kasue \cite{Kasue} and  Li-Tam \cite{Li-Tam} et. al.  studied manifolds with 
non-negative Ricci (sectional) curvature outside a compact set where manifolds with ends play an important role. 
It should be pointed out that there are other recent  works on manifolds with ends.  See, for instance, 
Carron \cite{Carron}, Doan \cite{Doan}, Duong, Li and  Sikora \cite{Duong}, Hassel, Nix and Sikora \cite{Riesz1},  Hassel and Sikora \cite{Riesz2}.

Because of the bottleneck structure inherent to most manifolds with ends, geometric and analytic properties of manifolds with ends are very different 
 from a manifold such as  $\mathbb{R}^n$. 
For example, in 1979, Kuz'menko and Molchanov \cite{Kuz'menko-Molchanov} proved the following:
\begin{theorem}
On $M=\mathbb{R}^3 \# \mathbb{R}^3$, the connected sum of two copies of $\mathbb{R}^3$, 
the weak Liouville property does not hold. Namely, 
there exists a non-trivial bounded harmonic function.
\end{theorem}
It is a well-known fact that the \textit{parabolic Harnack inequality} ((PHI) in short) implies the weak Liouville property.
See  \cite[Section 2.1]{G-SC stability} and  \cite[5.4.5]{SC LNS} for details. 
By a contraposition argument, the above theorem implies that (PHI) does not hold on 
$\mathbb{R}^3\# \mathbb{R}^3$. 

\hspace*{-3mm} Denote by $ p \!\left( t,x,y\right) $ the heat kernel of a non-compact weighted manifold $\! (M, d, \mu)$, that is, the
minimal positive fundamental solution of the heat equation $\partial
_{t}u=\Delta u$, where $\Delta$ is the weighed Laplacian. 
In 1986, Li and Yau proved in \cite{Li-Yau} that  
\begin{equation*}
(LY) \qquad \qquad  p(t,x,y) \asymp \frac{c}{V(x, \sqrt{t})} e^{-bd^2(x,y)/t}
\end{equation*}
 holds on non-compact manifold with non-negative Ricci curvature.
Here $V(x, r) := \mu(B(x,r))$, the measure of the open geodesic ball $B(x,r)=\{ y \in M ~:~ d(x,y)<r\}$ and 
the sign $\asymp$ means that both $\leq $ and $\geq$ hold but with different values of the positive 
constants $C$ and $b$. We call this estimate a Li-Yau type bound and write (LY) in short.
 The following theorem is a combined result of \cite%
{Grigoryan 1991}, \cite{SC 1992} based on previous contributions of Moser 
\cite{Moser}, Kusuoka--Stroock \cite{KS} et al.

\begin{theorem}
\label{TVD+PI}On a geodesically complete, non-compact weighted manifold $M$,
the following conditions are equivalent:

\noindent (1) The Li-Yau type heat kernel estimates (LY).

\noindent (2) The parabolic Harnack inequality (PHI).

\noindent (3) The \textit{Poincar\'e inequality} $(PI)$: there exists $C, \kappa >0$ such that for any $x \in M$, $r>0$ and 
$f \in C^\infty \left( \overline{B(x,r)} \right)$,
\begin{equation*}
(PI) ~~~\int_{B(x, r)} |f-f_{B(x,r)} |^2 d\mu \leq C \int_{B(x,  r)} | \nabla f |^2 d\mu,
\end{equation*}
where $f_{B(x,r)}=\frac{1}{V(x,r)} \int_{B(x,r)} f d\mu$,
 and the \textit{volume doubling property} $(VD)$: there exists $C>0$ such that for any $x\in M$, $r>0$,
\begin{equation*} 
V(x, 2r) \leq CV(x,r).
\end{equation*}
\end{theorem}

Combining the above two theorems, the connected sum 
$M=M_1\#M_2=\mathbb{R}^3 \# \mathbb{R}^3$ satisfies neither (PI) nor (LY). 
Indeed, the function
\begin{equation*}
f(x)=\left\{ 
\begin{array}{cl}
1 & x\in M_1, \\
-1& x \in M_2 
\end{array} 
\right.
\end{equation*}
implies that $f_{B(o, r)}=0$ for a central reference point $o \in M$ and
\begin{equation*}
\int_{B(o,r)} | f-f_{B(o,r)}|^2 d\mu \simeq r^3, \quad \int_{B(o,  r)} | \nabla f |^2 d\mu 
\simeq const,
\end{equation*}
which fails (PI). Here $f \simeq g$ means 
\begin{equation*}
c f \leq g \leq C f
\end{equation*}
with some positive constants $0 <c \leq C$ on a suitable range of functions $f$, $g$.
Moreover, Benjamini, Chavel and Feldman \cite{Benjamini-Chavel-Feldman} 
obtained in 1996 the following heat kernel estimate.
\begin{theorem}
For $n\geq 3$, let $M=M_1\# M_2=\mathbb{R}^n \# \mathbb{R}^n$. There exists $\varepsilon >0$ such that 
for $x \in M_1$, $y \in M_2$ with $|x| \simeq |y| \simeq \sqrt{t}$, 
\begin{equation*} 
p(t, x, y ) \leq \frac{1}{t^{\frac{n+\varepsilon}{2}}} \ll \frac{1}{t^{n/2}}.
\end{equation*}
\end{theorem}
This theorem asserts that the heat kernel between two different ends is significantly smaller 
than that on one end because of a 
\textit{bottleneck effect}.

In view of the above facts, it is natural to ask on a manifold with ends the behavior of the heat kernel 
$p(t,x,y)$ and 
the estimate of
\begin{equation}
\Lambda(B(x,r)):= 
\sup_{\substack{ f\in C^{1}(\overline{B(x,r)}) 
\\ f\neq \mathrm{const}}}
\frac{ \int_{B(x,r)} | f-f_{B(x,r)} |^2 d\mu}{ \int_{B(x,  r)} |\nabla f |^2 d\mu } ,
\label{PC}
\end{equation}
which is called the \textit{Poincar\'e constant}. 

\textbf{Notation.} 
Throughout this article, the letters $c, c^\prime, C, C^\prime, C^{\prime \prime}$ denote positive constants whose values 
may be different at different instances. When the value of a constant is significant, it will be explicitly stated.
\section{The state of the art}
\subsection{Setting}
First of all, we begin with the definition of what we call a manifold with finitely many ends.
For  a fixed integer $k\geq 2$, let $M_{1},...,M_{k}$ be a sequence of
geodesically complete, non-compact weighted manifolds of the same dimension.

\begin{definition}
We say that a weighted manifold $M$ is a manifold with $k$ ends $%
M_{1},M_{2},\ldots M_{k}$ and write 
\begin{equation}
M=M_{1}\#...\#M_{k}  \label{M1Mk}
\end{equation}%
if there is a compact set $K\subset M$ so that $M\setminus K$ consists of $k$
connected components $E_{1},E_{2},\ldots ,E_{k}$ such that each $E_{i}$ is
isometric (as a weighted manifold) to $M_{i}\setminus K_{i}$ for some
compact set $K_{i}\subset M_{i}$ (see Fig. \ref{figure: connectedsum}). Each 
$E_{i}$ (or $M_i$)  will be referred to as an \emph{end} of $M$.
\end{definition}
\vspace*{-20mm}
\begin{figure}[h]
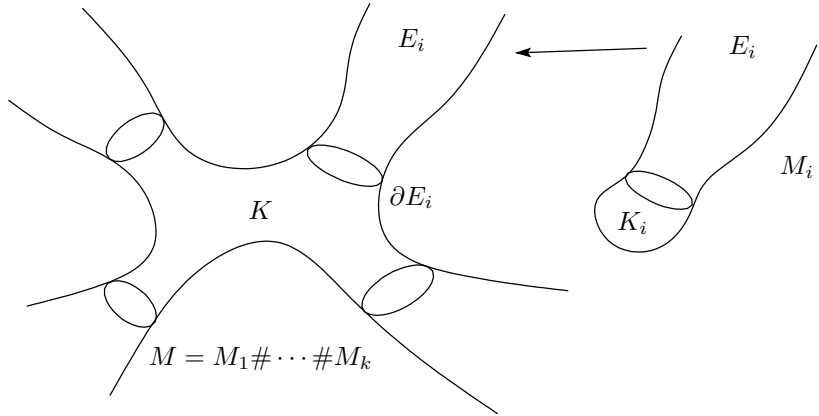

\hspace*{-20mm}
\vspace*{30mm}
\scalebox{1.8}{
\input connectedsum.tex
}
\caption{Manifold with ends}
\label{figure: connectedsum}
\end{figure}
Here we remark that the definition of end given above is different from the usual notion defined as a connected component of the
ideal boundary.

We say that a manifold $M$ is \textit{parabolic} if any positive superharmonic function on $M$ is constant, and 
\textit{non-parabolic} otherwise. See \cite{G 1999} for  details. 

Throughout  this article, we always assume that each end $M_i$ 
satisfies (VD) and (PI). Moreover, if  the end $M_i$ is parabolic, then we also assume that $M_i$ satisfies the  
\textit{relatively connected annuli} condition defined as follows.
\begin{definition}[(RCA)]
A weighted manifold $M$  
satisfies \textit{relatively connected annuli}  condition ((RCA) in short) 
with respect to a reference point $o\in M$ 
if there exists a positive constant $A>1$ 
such that for any $r>A^2$ and all $x,y \in M$ with $d(o, x)=d(o, y)=r$, there exists a continuous path from $x$ to $y$ 
staying in $B(o, Ar) \backslash B(o, A^{-1}r)$. See Fig. \ref{figure: RCA} and \ref{figure: RCA2} for typical positive and negative 
examples. 
\end{definition}

\begin{figure}[tbph]
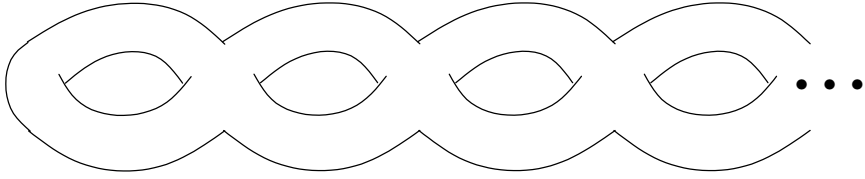

\begin{center}
\input RCA.tex
\end{center}
\caption{Manifold with (RCA)}
\label{figure: RCA}
\end{figure}

\begin{figure}[tbph]
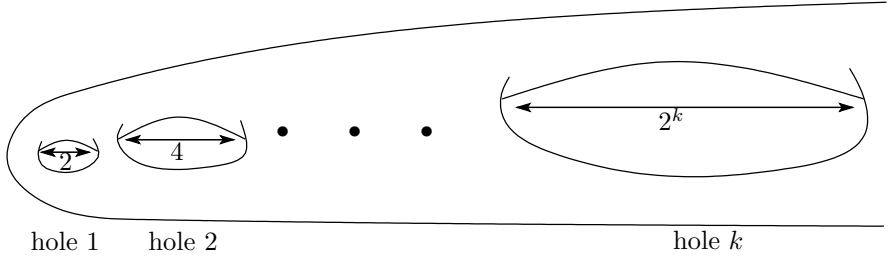

\begin{center}
\input RCA2.tex
\end{center}
\caption{Manifold without (RCA)}
\label{figure: RCA2}
\end{figure}

The assumption (RCA) seems technical but it makes it possible to obtain an optimal estimates of 
the first exit (hitting) probability and the Dirichlet heat kernel in the 
exterior of a compact set of a parabolic manifold satisfying (LY). See \cite{G-SC Dirichlet} and \cite{G-SC hitting} for  details.

\subsection{Heat kernel estimates}
\subsubsection{Off-diagonal estimates}
Let $M=M_1\# \cdots \# M_k$ be a manifold with $k$ ends. 
For  $t>0$, $x \in  E_i$ and $y \in M$, let $p^D_{E_i}(t, x, y)$ be the extended Dirichlet heat kernel on an end $E_i$, 
that is, the Dirichlet heat kernel in  $y \in E_i$ and extension to $0$ if $y \not \in E_i$.  
Let $\tau_{E_i}$ be the first exit time of the Brownian motion from $E_i$ and then 
$\mathbb{P}_x(\tau_{E_i}<t)$ is the 
first exit probability starting from $x$ by time $t$ from $E_i$.
We will use the following theorem to estimate the off-diagonal heat kernel estimates. 
\begin{theorem}[Grigor'yan and Saloff-Coste \mbox{\cite[Theorem 3.5]{G-SC ends}}]
\label{off-diagonal}
Let \\
$M=M_1 \# \cdots \# M_k$ be a manifold  with $k$ ends and fix a central reference point $o \in K$. 
For $x \in E_i$, $y \in E_j$ and $t>1$, 
\begin{align}
p(t, x, y) \simeq & p^D_{E_i}(t, x, y) +p(t, o, o)\mathbb{P}_x( \tau_{E_i} <t) \mathbb{P}_y( \tau_{E_j} <t)  \notag\\
 & \!\!\! +\int_0^t p(s, o, o)ds \left( \partial_t \mathbb{P}_x (\tau_{E_i}<t)  \mathbb{P}_y( \tau_{E_j} <t) 
+\mathbb{P}_x (\tau_{E_i}<t) \partial_t \mathbb{P}_{y} (\tau_{E_j} <t) \right).
\label{off-diagonal estimates}
\end{align}
\end{theorem}
Under the assumption of (PI), (VD) on each end, and, in addition,  (RCA) on each parabolic end, applying results in 
\cite{G-SC Dirichlet} and \cite{G-SC hitting}, 
the quantities 
$p^D_{E_i}(t, x, y)$, $\mathbb{P}_x(\tau_{E_i}<t)$, $\mathbb{P}_y(\tau_{E_j}<t)$, 
$\partial_t \mathbb{P}_x(\tau_{E_i})$ and $\partial_t \mathbb{P}_y (\tau_{E_j} <t)$ can be estimated. 
Hence, estimating $p(t, o, o)$ becomes the key missing estimate to obtain  off-diagonal bounds on manifolds with ends.

\subsubsection{Non-parabolic case}
First we consider heat kernel estimates on $M=M_1\# \cdots \# M_k$, where $M$ is non-parabolic, 
namely, at least one end $M_i$ is non-parabolic. 

For a fixed reference point $o_i \in K_i \subset M_i$, let 
\begin{equation*}
V_i(r):=\mu(B(o_i, r))
\end{equation*}
and 
\begin{equation*}
h_i(r):= 1+ \left(\int_1^r \frac{sds}{V_i(s)}\right)_{+}.
\end{equation*}
Here we remark that, under the assumption (LY), $M_i$ is parabolic if and only if
\begin{equation*}
\lim_{r\rightarrow \infty} h_i(r) =\infty.
\end{equation*}
In 2009,  Grigor'yan and Saloff-Coste \cite{G-SC ends} obtained the following (see also \cite{G-SC FK}).
\begin{theorem} Let $M=M_1 \# \cdots \# M_k$ be a manifold with $k$ ends.
Assume that each end $M_i$ satisfies (PI) and (VD) and that each parabolic end satisfies (RCA).
Assume also that $M$ is non-parabolic. 
Then for all $t>0$,
\begin{equation*}
p(t,o,o) \simeq \frac{1}{\min_i V_i(\sqrt{t}) h_i^2(\sqrt{t})} .
\end{equation*}
\end{theorem}
If all ends $M_1, \ldots , M_k$ are non-parabolic, then all functions $h_1, \ldots , h_k$ are bounded. Hence, the above theorem 
implies that
\begin{equation*}
p(t,o,o) \simeq \frac{1}{\min_i V_i(\sqrt{t})},
\end{equation*}
namely, the behavior of the heat kernel at the central reference point is determined by the \textbf{smallest end!}

As a typical example, let $M$ be $M_1\# M_2 =\mathbb{R}^n \#  \mathbb{R}^n$, the connected sum of two copies of $\mathbb{R}^n$ with 
$n\geq 3$. Then the above theorem 
implies that
\begin{equation*}
p(t, o, o) \simeq \frac{1}{t^{n/2}}.
\end{equation*}
Substituting this estimates into Theorem \ref{off-diagonal}, we obtain that for $x \in M_1$, $y \in M_2$, 
\begin{equation*}
p(t, x, y ) \simeq \frac{1}{t^{n/2}} \left( \frac{1}{| x|^{n-2}} +\frac{1}{|y|^{n-2}} \right) e^{-b d^2(x,y)/t},
\end{equation*}
where $|x|=1+d(x, K)$. 
\subsubsection{Parabolic case}
Next, we consider the case of  manifolds with ends, $M=M_1 \# \cdots \# M_k$, which are parabolic, that is, for which  
 all ends $M_1, \ldots , M_k$  are parabolic.
To prove an optimal heat kernel estimates, 
we need the following assumptions on each end.
\begin{definition}[c.f. \cite{GIS Poincare}]
\label{def sub}
An end  $M_i$ is called \textit{subcritical} if 
\begin{equation*}
h_i(r) \leq C\frac{r^2}{V_i(r)}\quad~(\forall r >1)
\end{equation*}
and \textit{regular} if there exist $\gamma_1, \gamma_2>0$ satisfying $2\gamma_1+\gamma_2 <2$ 
such that 
\begin{equation}
c\left( \frac{R}{r} \right)^{2-\gamma_2} \leq \frac{V_i(R)}{V_i(r)} \leq C \left( \frac{R}{r} \right)^{2+\gamma_1} \quad
(\forall 1<r \leq R).
\label{regular}
\end{equation}
\end{definition}
For example, a manifold $M$ with volume function $V(r)=r^{\alpha} \left( \log r \right)^{\beta}$ is parabolic if and only if
either $\alpha<2$ or $\alpha=2$ and $\beta\leq 1$. Moreover, 
$M$ is subcritical if $\alpha<2$ and regular if $\alpha=2$ and $\beta\leq 1$. 
We remark that if $M_i$ satisfies (VD), then the \textit{reverse doubling property} holds and that implies that  
for any subcritical end, 
there exists $\delta>0$ such that 
\begin{equation}
V_i(r) \leq Cr^{2-\delta} ~(\forall r>0). 
\label{sub delta}
\end{equation}

For $r>0$, let $m=m(r)$ be a number so that 
\begin{equation}
V_m(r) =\max_i V_i(r).
\label{largest}
\end{equation}
We can now state the following result.
\begin{theorem}[Grigor'yan, Ishiwata, Saloff-Coste  \cite{GIS Poincare}]
\label{theorem heat parab}
Let $M=M_1 \# \!\cdots \!\# M_k$ be a manifold with $k$ parabolic ends.
Assume that each end $M_i$ satisfies (PI), (VD), (RCA) and is either subcritical or regular.
If there exist both of subcritical and regular ends, assume also that the constant $\delta$ in (\ref{sub delta}) satisfies $\delta>\gamma_2$, namely, for any subcritical volume function  $V_i(r)$ and any regular volume function  $V_j(r)$, 
\begin{equation*}
V_i(r) \leq Cr^{2-\delta} \leq C^{\prime} r^{2-\gamma_2} \leq C^{\prime \prime} V_j (r) ~ (\forall r>0).
\end{equation*}
Moreover, assume that there exists an end $M_m$ such that 
for all $i=1, \ldots ,k$ and for all $r>0$
\begin{equation}
V_m(r) \geq c V_i(r) ~\mbox{and} ~ V_m(r)h_m^2(r) \leq C V_i(r) h_i^2(r).
\label{maximal}
\end{equation}
Then for $t>0$
\begin{equation}
p(t, o, o) \simeq \frac{1}{V_m(\sqrt{t})}.
\label{heat parab}
\end{equation}
\end{theorem}
This means that the on-diagonal heat kernel estimates at the central reference point 
is determined by the \textbf{largest end!}
\begin{remark}In our approach, we require the existence of  a fixed dominating end given by (\ref{maximal}) 
 for the optimal estimates in (\ref{heat parab}) to hold. 
Indeed, more generally, on a manifold with either regular or subcritical ends, we obtain for $t>1$, 
(see \cite{GIS Poincare} for the detail)
\begin{equation}
p(t, o, o) \leq C\frac{ \min_i h_i^2(\sqrt{t})}{\min_i V_i(\sqrt{t}) h_i^2 (\sqrt{t})}.
\label{min min}
\end{equation}
The assumption in (\ref{maximal}) implies that
for all $r>1$
\begin{equation}
\frac{ \min_i h_i^2(r)}{\min_i V_i(r) h_i^2 (r)} 
\leq \frac{C}{V_m(r)},
\label{min min 2}
\end{equation}
which allows  to apply \cite[Theorem 7.2]{Coulhon-Grigoryan} for the matching lower bound. 
In Section \ref{sec osci}, we construct manifolds with ends without a fixed dominating end and, in
such cases,  
the estimates in (\ref{min min 2}) does not hold.
\end{remark}
As illustrative examples, let 
$M=M_1\# \cdots \# M_k$ be a manifold with parabolic ends, where each end $M_i$  satisfies (PI), (VD), (RCA). 
Let $\alpha_i$ and $\beta_i$ be sequences satisfying 
\begin{equation*}
(\alpha_1, \beta_1) \geq (\alpha_2, \beta_2) \geq \cdots \geq  (\alpha_k, \beta_k) >(0, +\infty)
\end{equation*} 
in the sense of  lexicographical order, namely $(\alpha_i, \beta_i) > (\alpha_j, \beta_j)$ means that
\begin{equation*}
\alpha_i > \alpha_{j} ~ \mbox{or} ~ \alpha_i=\alpha_{j} ~\mbox{and} ~ \beta_i >  \beta_{j}
\end{equation*}
and we assume that 
\begin{equation*} 
V_i(r)\simeq r^{\alpha_i}\left( \log r \right)^{\beta_i}, \quad r>2.
\end{equation*}
Here we need $(\alpha_1, \beta_1) \leq (2,1)$ so that all ends $M_1, \ldots, M_k$ are parabolic. 
Then the above theorem implies that 
\begin{equation*}
p(t,o,o) \simeq \frac{1}{t^{\alpha_1/2} \left( \log t \right)^{\beta_1} }, \quad t>2.
\end{equation*}
As an explicit example, suppose that $k=2$ and $(\alpha_1, \beta_1)=(2,0)$ and $(\alpha_2, \beta_2)=(1,0)$.
Substituting the above estimates into (\ref{off-diagonal estimates}), we obtain for $x \in E_1$, $y \in E_2$ and $t>1$
\begin{equation*}
p(t,x,y)\simeq 
\left\{
\begin{array}{ll}
\frac{1}{t} e^{-bd^2(x,y)/t} & \mbox{if  } |x|>\sqrt{t} \\
\frac{1}{t} \left( 1+ \frac{|y|}{\sqrt{t}}\log\frac{e\sqrt{t}}{|x|} \right) & \mbox{if }  |x|, |y| \leq \sqrt{t} \\
\frac{1}{t} \left( \log \frac{e\sqrt{t}}{|x|} \right) e^{-bd^2(x,y)/t} & \mbox{if } |x| \leq \sqrt{t} < |y|.
\end{array}
\right.
\end{equation*}
\begin{remark}Assume that all ends of a manifold $M=M_1\# \cdots \# M_k$ are subcritical. Then for $x \in E_i$ and 
$y \in E_j$ with $i \neq j$ and $t>1$, 
\begin{equation*}
p(t,x,y) \asymp \frac{C}{V_m(\sqrt{t})} e^{-bd^2(x,y)/t}
\end{equation*}
(see \cite[Theorem 2.3]{GIS}).
\end{remark}

\subsection{Poincar\'e constant estimates}
In this section, we consider the estimates of the Poincar\'e constant defined in (\ref{PC}). Recall that $M=M_1 \# \cdots \# M_k$ 
is  a manifold with ends $M_1, \ldots , M_k$, where each end satisfies (VD) and (PI). Let $o \in K$ be a central reference point.
 Our main interest is to obtain the Poincar\'e constant $\Lambda(B(o,r))$ at the central point $o$. In fact, by the monotonicity 
of $\Lambda$ together with a Whitney covering argument (see \cite{GIS Poincare}), for $r>2|x|$
\begin{equation*}
\Lambda(B(x, r)) \simeq \Lambda (B(o, r)).
\end{equation*}

For $r>0$, let $n=n(r)$ be the number  so that 
\begin{align*}
V_n(r) &= \max_{i \neq m} V_i(r),
\end{align*}
where $m=m(r)$ is the number of the largest end (see (\ref{largest})).
Then we obtain the following. 
\begin{theorem}[Grigor'yan, Ishiwata, Saloff-Coste \cite{GIS Poincare}]
\label{theorem PC non-parab}
 Let $M=M_1 \# \!\cdots \!\# M_k$ be a manifold with $k$ non-parabolic ends. 
Assume that each end $M_i$ satisfies (VD) and  (PI). 
Then for sufficiently large $r>1$
\begin{equation*}
\Lambda(B(o,r)) \leq CV_n(r).
\end{equation*}
Moreover, if for all $r>0$
\begin{equation*}
rV^\prime (r) \leq C V(r) ,
\end{equation*}
 then for sufficiently large $r> 1$
\begin{equation*}
\Lambda(B(o,r)) \simeq V_n(r).
\end{equation*}
\end{theorem}
When $M$ has at least one parabolic end, we assume the following additional condition (see \cite{GIS Poincare} for details).
\begin{definition}[(COE)]
\label{DefCOE}We say that a manifold $M=\#_{i\in I}M_{i}$ has \emph{%
critically ordered ends} and write (COE) in short  if there exist $\varepsilon ,\delta ,\gamma
_{1},\gamma _{2}>0$ such that 
\begin{equation*}
\gamma _{1}<\varepsilon ,~~\gamma _{1}+\gamma _{2}<\delta <2,~~2\gamma
_{1}+\gamma _{2}<2,  
\end{equation*}%
and a decomposition 
\begin{equation*}
I=I_{super}\sqcup I_{middle}\sqcup I_{sub}  
\end{equation*}%
such that the following conditions are satisfied:

\begin{itemize}
\item[$\left( a\right) $] For each $i\in I_{super}$ and all $r\geq 1$, 
\begin{equation*}
V_{i}(r)\geq cr^{2+\epsilon }\ .
\end{equation*}

\item[$\left( b\right) $] For each $i\in I_{sub}$, $V_{i}$ is subcritical (see Definition \ref{def sub}) 
and 
\begin{equation*}
V_{i}(r)\leq C r^{2-\delta }\ .  
\end{equation*}

\item[$\left( c\right) $] For each $i\in I_{middle}$, $V_i$ is regular (see (\ref{regular})). 
Moreover, for any pair $i,j\in I_{middle}$ we have either $V_{i}\geq c
V_{j} $ or $V_{j}\geq  c V_{i}$ (i.e., the ends in $I_{middle}$ can be
ordered according to their volume growth uniformly over $r\in \lbrack
1,\infty )$) and $V_{i}\geq c V_{j}$ implies that $V_{i}h_{i}\geq c^\prime 
V_{j}h_{j}$.\ Besides, if $M$ is parabolic (i.e., all ends are parabolic)
then $V_{i}\geq C V_{j}$ also implies $V_{i}h_{i}^{2}\leq C^\prime  V_{j}h_{j}^{2}
$.
\end{itemize}
\end{definition}
\begin{theorem}[\cite{GIS Poincare}]\label{theorem PC COE}
Let $M=M_1 \# \cdots \# M_k$ be a manifold with $k$ ends, where each end satisfies (VD) and (PI). 
Suppose that there exists at least one parabolic end and each parabolic end satisfies (RCA). 
If $M$ admits (COE), then for sufficiently large $r>1$
\begin{equation*}
\Lambda(B(o,r)) \leq CV_n(r) h_n(r).
\end{equation*}
If, in addition, each $V_i$ satisfies that
\begin{equation*}
rV_i^\prime (r) \leq C V_i(r) ~~ (\forall r>1),
\end{equation*}
then, for sufficiently large $r> 1$
\begin{equation*}
\Lambda(B(o,r)) \simeq V_n(r) h_n(r).
\end{equation*}
\end{theorem}
These results say that the Poincar\'e constant $\Lambda(B(o,r))$ is determined by the  \textbf{second largest end!}

As an explicit example, let $M=\mathbb{R}^n \# \mathbb{R}^n$ with $n\ge 2$. Then Theorems 
\ref{theorem PC non-parab} and \ref{theorem PC COE} imply that 
\begin{equation*}
\Lambda (B(o,r)) \simeq \left\{ 
\begin{array}{ll}
 r^n & n \geq 3, \\
r^2 \log r & n=2.
\end{array}
\right.
\end{equation*}

Let $M=M_1\# \cdots \#M_k$ be a manifold with ends. Assume that each end $M_i$ satisfies (VD), (PI) and that 
each parabolic end satisfies (RCA). 
Suppose that for $i=1, \ldots k$
\begin{equation*}
V_i(r) \simeq r^{\alpha_i} \left( \log r \right)^{\beta_i},
\end{equation*}
where $(\alpha_1, \beta_1 ) \geq (\alpha_{2}, \beta_{2}) \geq \cdots \geq (\alpha_k, \beta_k)$ as the lexicographical order. Then 
\begin{align*}
\Lambda(B(o,r)) &\simeq V_2(r)h_2(r) \\
& \simeq
\left\{
\begin{array}{ll}
r^{\alpha_2} (\log r)^{\beta_2}&\mbox{if }  (\alpha_2, \beta_2)>(2,1) ,\\
r^2\log r (\log \log r )^2 &\mbox{if } (\alpha_2, \beta_2)=(2,1), \\
r^2\log r & \mbox{if } (2, -\infty)< (\alpha_2, \beta_2)<(2, 1),\\
r^2 &\mbox{if } (\alpha_2, \beta_2)<(2, -\infty).
\end{array}
\right.
\end{align*}

\section{Manifold with ends with oscillating volume functions}
\label{sec osci}
\subsection{Preliminaries}
The purpose of this section is to construct manifolds with ends for which the estimate in (\ref{min min}) 
might not give an optimal bound. To obtain such a manifold, we need a manifold with (VD) and (PI) 
together with oscillating volume function.
First, let us recall the following theorem.
\begin{theorem}[Grigor'yan and Saloff-Coste \mbox{\cite[Theorem 5.7]{G-SC stability}}]
\label{W}
Let $(M, \mu)$ be a complete non-compact wighted manifold with (PHI) and (RCA) at a reference point $o \in M$. 
If a positive valued smooth function 
$W : [0, \infty)\rightarrow \mathbb{R}$ satisfies for all $r>0$
\begin{align}
\sup_{ [r, 2r] } W & \leq C \inf_{ [r,2r] } W,  \label{h1}\\
\int_0^r W^2(s)s ds & \leq C W^2 (r) r^2, \label{h2}
\end{align}
then the weighted manifold $(M, W^2(d(o, \cdot) ) \mu)$ also satisfies (PHI).
\end{theorem}

Let $(M_1, \mu_1)$ be the $2$-dimensional Euclidean space $\mathbb{R}^2$ with the Euclidean measure. 
We denote by $(M_2, \mu_2)$ a weighted manifold $(\mathbb{R}^2, W^2(d(o,\cdot)) \mu_1)$, where the positive valued 
smooth function $W:[0, \infty) \rightarrow \mathbb{R}$ 
is defined as follows. For $\alpha>2$, $0<\beta< 2$ define a function $W$ so that for all $k \in \mathbb{N}$
\begin{equation}
\int_0^r W^2 (s) s ds =\left\{
\begin{array}{ll}
r^2 & 0<r<a_1, a_k \leq r < b_k, \\
\left( \frac{r}{b_k} \right)^{\alpha} b_k^2 & b_k \leq r <c_k, \\
r^2 \log r & c_k \leq r < d_k ,\\
\left( \frac{r}{d_k} \right)^{\beta} d_k^2 \log d_k & d_k \leq r <a_{k+1},
\end{array} 
\right. 
\label{def W}
\end{equation}
where the sequences $a_k \leq b_k < c_k \leq d_k < a_{k+1}$ satisfy $a_1> e$ and 
\begin{align}
b_k &= \frac{c_k}{ (\log c_k)^{\frac{1}{\alpha-2}}} ,  \label{bc} \\
a_{k+1} &= d_k (\log d_k)^{\frac{1}{2-\beta} } \label{ad}
\end{align}
 (see fig. \ref{figure: osc1}).
 These sequences  will be fixed later.
\begin{figure}[ht]
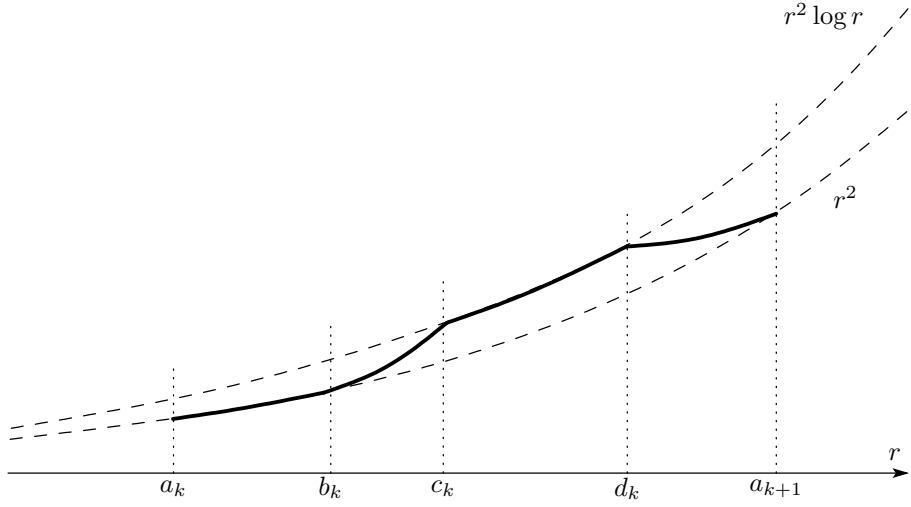

\input osc1.tex
\caption{Oscillating volume function (thick line)}
\label{figure: osc1}
\end{figure}
Then the function $W$ satisfies that 
\begin{align*}
W(r)\simeq\left\{
\begin{array}{ll}
1 & 0<r<a_1,  a_k \leq r < b_k, \\
\left( \frac{r}{b_k} \right)^{\frac{\alpha-2}{2}} & b_k \leq r <c_k, \\
\sqrt{ \log r } & c_k \leq r < d_k, \\
\sqrt{ \log d_k } \left( \frac{r}{d_k} \right)^{-\frac{2-\beta }{2}} & d_k \leq r < a_{k+1},
\end{array}
\right.
\end{align*}
\begin{align*}
W^\prime (r)\simeq\left\{
\begin{array}{ll}
0 & 0<r<a_1, a_k < r < b_k, \\
r^{\frac{\alpha-4}{2}}b_k^{-\frac{\alpha-2}{2}} & b_k < r <c_k, \\
\frac{1}{\sqrt{ \log r }} \frac{1}{r} & c_k < r < d_k, \\
-r^{-\frac{4-\beta}{2}} \sqrt{ \log d_k } d_k^{\frac{2-\beta}{2}} & d_k < r < a_{k+1}.
\end{array}
\right.
\end{align*}

Since
\begin{align}
\frac{W^\prime (r)}{W(r)} \simeq\left\{
\begin{array}{ll}
0 & 0<r<a_1, a_k < r < b_k, \\
\frac{1}{r} & b_k < r <c_k, \\
\frac{1}{r\log r } & c_k < r < d_k, \\
-\frac{1}{r} & d_k < r < a_{k+1},
\end{array}
\right.
\label{ratio}
\end{align}
the condition in (\ref{h1}) holds. Indeed, the estimates in (\ref{ratio}) imply that
 for any $0<r_1<r_2$
\begin{equation*}
-C\log \frac{r_2}{r_1}=-\int_{r_1}^{r_2} \frac{Cds}{s}  \leq \int_{r_1}^{r_2} \frac{W^\prime(s)}{W(s)} ds 
=\log \frac{W(r_2)}{W(r_1)}
\leq \int_{r_1}^{r_2} \frac{C ds}{s}=C\log \frac{r_2}{r_1}.
\end{equation*}
Then we obtain for any $0<r_1 \leq r_2 \leq 2r_1$
\begin{equation*}
 W(r_1) \simeq W(r_2). 
\end{equation*}
Taking $r_1\leq r_2, r_3 \leq 2r_1$ so that 
\begin{equation*}
W(r_2)=\sup_{[r_1, 2r_1]} W \mbox{ and } W(r_3)=\inf_{[r_1, 2r_1]}W, 
\end{equation*}
we conclude the condition in (\ref{h1}).

Moreover, we see that 
\begin{align*}
W^2 (r) r^2 \simeq\left\{
\begin{array}{ll}
r^2 & 0<r<a_1, a_k \leq r < b_k, \\
r^{\alpha} b_k^{2-\alpha} & b_k \leq r <c_k, \\
r^2 \log r  & c_k \leq r < d_k, \\
r^{\beta} d_k^{2-\beta} \log d_k & d_k \leq r < a_{k+1},
\end{array}
\right.
\end{align*}
which satisfies the condition in (\ref{h2}). Applying Theorem \ref{W}, the weighted manifold 
$(M_2,\mu_2)=(\mathbb{R}^2, W^2(d(o, \cdot)) \mu_1)$ 
satisfies (PHI)
\footnote{
We need a smooth modification of the function $W$ satisfying (\ref{def W}) to apply Theorem \ref{W}. 
However, we omit the smoothing argument for simplicity.}
. 

Let $Z: [0, \infty) \rightarrow \mathbb{R}$ be a positive valued smooth function satisfying 
\begin{equation*}
Z(r)=\sqrt{1+2\log r} \quad (r\geq 1).
\end{equation*}
Let $(M_3, \mu_3)$ be a weighted manifold $(\mathbb{R}^2, Z^2(d(o, \cdot)) \mu_1)$.
Then  Theorem \ref{W} implies also that $(M_3, \mu_3)$ satisfies (PHI).

For $i=1,2,3$, we denote by $V_i(r)$ the volume function on $(M_i, \mu_i)$ at $o \in M_i=\mathbb{R}^2$. 
Then  we obtain 
\begin{align*}
V_1 (r) &=\pi r^2,\\
V_2(r)&=2\pi \int_0^r W^2(s) sds ,\\
V_3(r)&=2\pi \int_0^1 Z^2(s)sds+ 2\pi r^2 \log r ~~(r \geq 1).
\end{align*}
It is  easy to obtain for sufficiently large $r>1$
\begin{equation*}
h_1(r)\simeq \log r ~\mbox{ and }~ h_3(r)\simeq \log \log r.
\end{equation*}

Now we estimate the function $h_2(r)$. Observe that
\begin{align*}
\int_{a_k}^{b_k} \frac{sds}{V_2(s)} & =\log \frac{b_k}{a_k}, \\
\int_{b_k}^{c_k} \frac{sds}{V_2(s)} &= \int_{b_k}^{c_k} b_k^{\alpha-2} s^{1-\alpha} ds 
= \frac{1}{\alpha-2} \left( 1- \frac{1}{\log c_k} \right)  \simeq 1,  \\
\int_{c_k}^{d_k} \frac{sds}{V_2(s)} &= \int_{c_k}^{d_k} \frac{ds}{s \log s} =\log \left( \frac{ \log d_k}{\log c_k} \right), \\
\int_{d_k}^{a_{k+1}} \frac{sds}{V_2(s)} &= \int_{d_k}^{a_{k+1}} \frac{s^{1-\beta}ds}{d_k^{2-\beta} \log d_k} = \frac{1}{(2-\beta) } \left( 1- \frac{1}{ \log d_k} \right) \simeq 1.
\end{align*}
Then we obtain
\begin{equation}
h_2(a_{n}) =1+\int_{1}^{a_1} \frac{sds}{V_2(s)} +\sum_{k=1}^{n-1} \int_{a_k}^{a_{k+1}} \frac{sds}{V_2(s)} \simeq \sum_{k=1}^n 
\left( \log \frac{b_k}{a_k} + \log \frac{\log d_k}{\log c_k} +1 \right),
\label{estimate h2}
\end{equation}
which shows that the behavior of $h_2(r)$ depends on the choice of sequences $a_k\leq b_k< c_k \leq d_k<a_{k+1}$
satisfying (\ref{bc}) and (\ref{ad}).

\subsection{Example 1}
For the first case, let us choose sequences $a_k \leq b_k < c_k \leq d_k<a_{k+1}$ so that for any $k\in \mathbb{N}$
\begin{equation}
a_k \simeq  b_k, \quad c_k \simeq d_k.
\label{case 1}
\end{equation}
In this case, we obtain the following.
\begin{lemma}\label{lemma h2}
If $a_1$ is large enough and the sequences $a_k\leq b_k< c_k  \leq d_k <a_{k+1} $ 
satisfy  (\ref{bc}), (\ref{ad}) and (\ref{case 1}), then for sufficiently large $r>1$, 
\begin{equation*}
h_2(r) \simeq \frac{ \log r}{\log \log r}.
\end{equation*}
\end{lemma}
\begin{proof}
By the estimate in (\ref{estimate h2}), we obtain
\begin{equation}
h_2(a_n) \simeq n.
\label{ann}
\end{equation}
Let us consider the behavior $a_n$. 
By the assumption in (\ref{bc}), we always have
\begin{equation}
c_k \simeq b_k \left( \log b_k \right)^{\frac{1}{\alpha-2}}. \label{bkck}
\end{equation}
Assumptions in (\ref{bc}) and  (\ref{ad}) imply that for any $k\in \mathbb{N}$
\begin{align*}
a_{k+1} & = d_k \left( \log d_k \right)^{\frac{1}{2-\beta}} \simeq c_k \left( \log c_k \right)^{\frac{1}{2-\beta}}\\
& \simeq b_k   \left( \log b_k \right)^{\frac{1}{\alpha-2}} 
\left( \log \left( b_k \left( \log b_k \right)^{\frac{1}{\alpha-2}} \right) \right)^{\frac{1}{2-\beta}}\\
&=b_k  \left( \log b_k \right)^{\frac{1}{\alpha-2}} 
\left( \log b_k + \frac{1}{\alpha-2} \log \log b_k \right)^{\frac{1}{2-\beta}} \\
&\simeq b_k \left( \log b_k  \right)^{\frac{1}{\alpha-2} + \frac{1}{2-\beta}} 
 = b_k \left( \log b_k \right)^{\gamma} \simeq a_k \left( \log a_k \right)^{\gamma},
\end{align*}
where $\gamma=\frac{1}{\alpha-2} +\frac{1}{2-\beta}$. Taking $a_1$ large enough, 
there exist positive constants $c, C$ and $\gamma^\prime >\gamma$ such that for any $n \in \mathbb{N}$
\begin{equation*}
c n^{\gamma n}  \leq a_n \leq Cn^{\gamma^\prime n}.
\end{equation*}
This implies that  for sufficiently large $n \in \mathbb{N}$
\begin{align*}
\log a_n \simeq n \log n,\quad 
\log \log a_n \simeq  \log n.
\end{align*}
Then we obtain for sufficiently large $n \in \mathbb{N}$,
\begin{equation*} 
 \frac{\log a_n}{\log \log a_n}\simeq n \simeq h_2(a_n),
\end{equation*}
which concludes the lemma.
\end{proof}

Now we estimate the heat kernel on $M=M_1\# M_2$. 
By the definition of $V_1(r)$, $V_2(r)$ and by Lemma \ref{lemma h2}, we obtain for sufficiently large $r>1$
\begin{align*}
V_1(r)&=r^2,  &&V_2(r)\simeq\left\{ 
\begin{array}{ll}
r^2 & a_k \leq r <b_k \\
r^2 \log r & c_k \leq r <d_k,
\end{array}
\right. \\
h_1(r) &\simeq \log r,  && h_2(r) \simeq \frac{\log r}{\log \log r} ,\\
V_1(r)h_1^2(r)&\simeq  r^2(\log r)^2,  &&V_2(r)h_2^2(r) \simeq \left\{ 
\begin{array}{ll}
r^2 \frac{ (\log r )^2}{(\log \log r )^2} & a_k \leq r < b_k \\
r^2 \frac{(\log r)^3}{(\log \log r)^2} & c_k \leq r <d_k.
\end{array}
\right.
\end{align*}
According to the heat kernel upper estimate in (\ref{min min}), we obtain  for $c_k \leq \sqrt{t} < d_k$
\begin{equation*}
p(t, o, o) \leq \frac{ \min_i  h_i^2 (\sqrt{t})}{ \min_i V_i(\sqrt{t}) h_i^2 (\sqrt{t})} 
\simeq \frac{ \left( \frac{ \log \sqrt{t} }{\log \log \sqrt{t}} \right)^2}{ t \left( \log \sqrt{t} \right)^2 }
\simeq \frac{1}{ t (\log \log t )^2}.
\end{equation*}
Since $V(r) \simeq \max_i V_i(r) =r ^2\log r$ in this interval, the above upper estimate 
is much larger than 
\begin{equation*}
\frac{1}{V(\sqrt{t})}\simeq \frac{1}{t \log t},
\end{equation*}
which makes difficult to obtain matching lower bound by using \cite[Theorem 7.2]{Coulhon-Grigoryan}.

\subsection{Example 2}
Next, let us choose sequences $a_k \leq b_k< c_k \leq d_k< a_{k+1}$ so that for some $\delta>1$ and for any $k \in \mathbb{N}$
\begin{equation}
a_k\simeq b_k, \quad c_k^{\delta} \simeq d_k. \label{case 2}
\end{equation}
In this case, we obtain the following.
\begin{lemma}\label{lemma case 2}
If $a_1$ is large enough and the sequences $a_k \leq b_k<c_k \leq d_k<a_{k+1}$ 
satisfy (\ref{bc}),  (\ref{ad}) and (\ref{case 2}), then for sufficiently large $r>1$, 
\begin{equation*}
h_2(r) \simeq \log \log r.
\end{equation*}
\end{lemma}
\begin{proof}
The estimate in (\ref{estimate h2}) yields that 
\begin{equation*}
h_2(a_n) \simeq n.
\end{equation*}
By the estimates in (\ref{bkck}) and (\ref{case 2}), we obtain
\begin{align*}
d_k &\simeq c_k^{\delta}
\simeq  \left( b_k \left( \log b_k \right)^{\frac{1}{\alpha-2}} \right)^{\delta} 
\simeq a_k^{\delta} \left( \log a_k \right)^{\frac{\delta}{\alpha-2}}.
\end{align*}
Assumptions in (\ref{bc}) and (\ref{ad}) imply that
\begin{align*}
a_{k+1} &= d_k \left( \log d_k \right)^{\frac{1}{2-\beta}} \\
&\simeq a_k^{\delta} \left( \log a_k \right)^{\frac{\delta}{\alpha-2}} 
\left( \delta \log a_k +\frac{\delta}{\alpha -2} \log\log a_k \right)^{\frac{1}{2-\beta}}\\
&\simeq a_k^{\delta} \left( \log a_k \right)^{\theta},
\end{align*}
where $\theta =\frac{\delta}{\alpha-2} + \frac{1}{2-\beta}$. Hence, there exist positive constants $c, C$ 
and $\eta > \delta$ such that for any $k \in \mathbb{N}$
\begin{equation*}
ca_k^{\delta} \leq a_{k+1} \leq Ca_k^{\eta}.
\end{equation*}
This implies that for any $n \in \mathbb{N}$
\begin{equation*}
c^{\frac{\delta^{(n-1)}-1}{\delta-1}}a_1^{\delta^{(n-1)}}  \leq a_n \leq 
C^{\frac{\eta^{(n-1)}-1}{\eta-1}} a_1^{\eta^{(n-1)}}.
\end{equation*}
Then we obtain
\begin{align*}
\frac{\delta^{n-1}-1}{\delta-1} \log c + \delta^{(n-1)} \log a_1 
\leq \log a_n \leq \frac{\eta^{(n-1)}-1}{\eta-1} \log C + \eta^{(n-1)} \log a_1.
\end{align*}
Taking $a_1$ large enough, we obtain for sufficiently large $n\in \mathbb{N}$
\begin{equation*}
\log \log a_n \simeq n \simeq h_2(a_n),
\end{equation*}
which concludes the lemma.
\end{proof}

Now let us consider the estimate of the heat kernel on $M=M_2 \# M_3$. By the definition of $V_2(r)$, $V_3(r)$ 
and by Lemma \ref{lemma case 2}, we obtain for sufficiently large $r>1$
\begin{align*}
V_2(r)&\simeq \left\{ 
\begin{array}{ll} 
r^2 & a_k \leq r < b_k, \\
r^2 \log r & c_k \leq r <d_k ,
\end{array}
\right. &&\!\!\!\! V_3(r)\simeq r^2 \log r, \\
h_2(r) &\simeq \log \log r, &&\!\!\!\! h_3(r) \simeq \log \log r , \\
V_2(r)h_2^2(r) &\simeq 
\left\{ 
\begin{array}{ll} 
r^2(\log \log r)^2 & a_k \leq r < b_k, \\
r^2 \log r (\log \log r)^2 & c_k \leq r <d_k .
\end{array}
\right. 
&&\!\!\!\! V_3(r)h_3^2 (r) \simeq r^2 \log r (\log \log r)^2.
\end{align*}
Substituting above into  (\ref{min min}), we obtain for sufficiently large $k \in \mathbb{N}$ and  $a_k \leq \sqrt{t} < b_k$
\begin{equation*}
p(t, o, o) \leq C \frac{ (\log \log \sqrt{t})^2}{t (\log \log \sqrt{t} )^2} \simeq \frac{1}{t}.
\end{equation*}
Since $V(r) \simeq \max_i V_i(r) \simeq  r^2 \log r$ for all $r>1$,  the above upper estimate is much larger than
\begin{equation*}
\frac{1}{V(\sqrt{t})} \simeq \frac{1}{t \log t},
\end{equation*}
which makes difficult to obtain matching lower bound by using \cite[Theorem 7.2]{Coulhon-Grigoryan}.

We hope to prove matching heat kernel lower bounds in forthcoming work. 

\section*{Acknowledgments}
The second author would like to thank Professor Gilles Carron for telling him how to construct manifolds with 
oscillating volume function with (PHI).

\end{document}